\documentclass[12pt]{article}
\usepackage{setting0730}

\renewcommand{\ArticleHeaderText}{Zhengxu Jiang, B-coloring of \(K_{2,t}\)-free planar graphs}
\title{B-coloring of \(K_{2,t}\)-free planar graphs}
\author[1]{Zhengxu Jiang}
\affil[1]{School of Mathematical Sciences and LPMC, Nankai University, Tianjin 300071, China}
\date{}

\hypersetup{
  pdftitle={B-coloring of K2,t-free planar graphs},
  pdfauthor={Zhengxu Jiang},
  pdfkeywords={B-coloring, K2,t-free graph, planar graph, degenerate graph}
}

\newcommand{\qB}{q_B}

\newtheorem*{mainproblem}{Problem}
\newtheoremstyle{mainresult}%
  {6pt}{6pt}%
  {\normalfont}%
  {0pt}%
  {\bfseries}%
  {.}%
  {0.6em}%
  {}%
\theoremstyle{mainresult}
\newtheorem*{theorema}{Theorem A}
\newtheorem*{theoremb}{Theorem B}

\begin{document}
\maketitle

\begin{abstract}
A B-coloring of a graph \(G\) is a proper edge-coloring in which every \(4\)-cycle receives four distinct colors; let \(\qB(G)\) be the minimum number of colors in such a coloring. Every graph of maximum degree \(\Delta\) is \(K_{2,\Delta+1}\)-free; hence the known \(2\Delta\) bound for planar graphs with \(\Delta\ge38\) (Kong et al., 2026) motivates our study of \(K_{2,t}\)-free planar graphs, where \(t\ge2\) is an integer. We prove \(\qB(G)=\Delta(G)\) when \(t=2\) and \(\Delta(G)\ge7\), or when \(t\ge3\) and \(\Delta(G)\ge14(t-1)\). For \(t\ge35\), the bound \(\qB(G)\le\Delta(G)+t-1\) holds regardless of \(\Delta(G)\); for every \(t\ge2\), it also holds when \(\Delta(G)>428\). Finally, for every integer \(k\ge1\), every \(k\)-degenerate \(K_{2,t}\)-free graph satisfies \(\qB(G)\le\Delta(G)+(k-1)\min\{t-1,\Delta(G)\}\), with equality for \(K_{k,t-1}\) when \(k\ge2\) and \(t-1\ge k\).

\end{abstract}

\section{Introduction}

All graphs in this paper are finite, simple, and undirected. An edge-coloring is \emph{proper} if incident edges receive distinct colors, and a set of edges is \emph{rainbow} if its edges have pairwise distinct colors. A \emph{B-coloring} of a graph \(G\) is a proper edge-coloring in which every \(4\)-cycle is rainbow; the cycles need not be induced. The \emph{B-coloring number} \(\qB(G)\) is the minimum number of colors in a B-coloring of \(G\). Necessarily, \(\qB(G)\ge\Delta(G)\). Let \(F(G)\) be the graph with vertex set \(E(G)\), in which two edges are adjacent when they are opposite edges of a \(4\)-cycle, and let \(L^+(G)=L(G)\cup F(G)\)~\cite{GyarfasMartinRuszinkoSarkozy2024}. Thus a B-coloring of \(G\) is exactly a proper vertex-coloring of \(L^+(G)\), and \(\qB(G)=\chi(L^+(G))\).

Gy\'arf\'as and S\'ark\"ozy introduced B-coloring partly through its connection with the famous \((7,4)\)-problem of Brown, Erd\H{o}s, and S\'os, which asks whether every triple system on \(n\) points containing no four triples on seven points has \(o(n^2)\) triples~\cite{GyarfasSarkozy2023}. For a planar graph \(G\) with maximum degree \(\Delta\), edge-coloring requires only \(\Delta\) colors when \(\Delta\ge7\), in which case \(G\) is of class~1~\cite{SandersZhao2001,Zhang2000}. With the additional requirement that every \(4\)-cycle be rainbow, Gy\'arf\'as, Martin, Ruszink\'o, and S\'ark\"ozy proved \(\qB(G)\le2\Delta+8\)~\cite{GyarfasMartinRuszinkoSarkozy2024}, which Kong, Wang, and Zheng improved to \(2\Delta+6\), and further to \(2\Delta+4\) when \(\Delta\ge12\) and \(2\Delta\) when \(\Delta\ge38\)~\cite{KongWangZheng2026}.

For an integer \(t\ge2\), a graph is \emph{\(K_{2,t}\)-free} if any two distinct vertices have at most \(t-1\) common neighbors. \(K_{2,t}\)-free graphs arise in extremal graph theory through the Zarankiewicz problem and the K\H{o}v\'ari--S\'os--Tur\'an theorem~\cite{KovariSosTuran1954}, and have also been studied in work on star edge coloring, which is closely related to B-coloring~\cite{TangYinHan2023}. Since every graph of maximum degree \(\Delta\) is \(K_{2,\Delta+1}\)-free, the known \(2\Delta\) bound suggests the more general inequality
\begin{equation}
\qB(G)\le\Delta(G)+t-1.
\label{eq:target-palette}
\end{equation}
Since the \(2\Delta\) bound is known for \(\Delta\ge38\), we ask whether a sufficiently large maximum degree also guarantees the proposed bound.

\begin{mainproblem}
For every integer \(t\ge2\), does there exist a constant \(M(t)\) such that every \(K_{2,t}\)-free planar graph \(G\) with maximum degree \(\Delta>M(t)\) satisfies \(\qB(G)\le\Delta+t-1\)? If so, can \(M(t)\) be bounded by an absolute constant independent of \(t\)?
\end{mainproblem}

Our first result determines the B-coloring number when the maximum degree is large relative to \(t\).

\begin{theorema}
Let \(t\ge2\) be an integer, and let \(G\) be a \(K_{2,t}\)-free planar graph with maximum degree \(\Delta\). If \(t=2\) and \(\Delta\ge7\), or if \(t\ge3\) and \(\Delta\ge14(t-1)\), then \(\qB(G)=\Delta\).
\end{theorema}

Our second result shows that \(M(t)=428\) suffices for every \(t\ge2\).

\begin{theoremb}
Let \(t\ge2\) be an integer. Every \(K_{2,t}\)-free planar graph \(G\) with maximum degree \(\Delta>428\) satisfies \(\qB(G)\le\Delta+t-1\).
\end{theoremb}

Theorem~A gives the exact value under a lower bound on \(\Delta\) that depends on \(t\), whereas Theorem~B proves \(\qB(G)\le\Delta+t-1\) for every \(t\) once \(\Delta>428\). Sharper bounds for individual values of \(t\) are stated in \cref{thm:fixed-t-thresholds}. The proof of Theorem~A shows how the assumed lower bound on \(\Delta\) controls the number of colors. For \(t\ge35\), a structural theorem for plane graphs removes that lower bound; together, the two arguments yield the absolute bound \(428\). The vertex-extension argument also yields the degenerate-graph results in the final section.

\Cref{sec:theorem-a} proves Theorem~A by coloring the edges in an order determined by the vertices; it also gives the bound used for \(3\le t\le34\) in \cref{sec:theorem-b}. That section treats \(t\ge35\) and then proves Theorem~B. \Cref{sec:degenerate} proves the \(k\)-degenerate bound.

\section{Proof of Theorem~A}
\label{sec:theorem-a}

We first prove Theorem~A by coloring the edges in an order determined by the vertices. The same estimates will be used for \(3\le t\le34\) in the proof of Theorem~B.

\subsection{Vertex extension and two-step orders}

We use the following vertex-extension lemma. Let \(v\) have neighbors \(x_1,\ldots,x_r\), and suppose that \(G-v\) has a B-coloring \(\varphi\) from a fixed color set \(C\). For \(1\le i\le r\), let \(L_i\) consist of the colors in \(C\) that appear neither on an edge of \(G-v\) incident with \(x_i\) nor on an edge \(x_jy\), where \(j\ne i\) and \(y\in N_G(x_i)\cap N_G(x_j)\setminus\{v\}\). Thus \(L_i\) is the set of colors that may be assigned to \(vx_i\) without repeating a color at \(x_i\) or on a new \(4\)-cycle.

\begin{lemma}[Vertex extension]
\label{lem:three-source}
If one can choose pairwise distinct colors \(c_i\in L_i\) for \(1\le i\le r\), then setting \(\varphi(vx_i)=c_i\) extends \(\varphi\) to a B-coloring of \(G\).
\end{lemma}

\begin{proof}
The extension is proper at every \(x_i\), and the new edges receive pairwise distinct colors at \(v\). Every old \(4\)-cycle remains rainbow. A new \(4\)-cycle has the form \(vx_iyx_jv\), as shown in \cref{fig:vertex-extension}. By the definition of the lists, \(c_i\ne\varphi(x_jy)\) and \(c_j\ne\varphi(x_iy)\). The two old edges have different colors because they meet at \(y\), and \(c_i\ne c_j\). Hence the new cycle is rainbow.
\end{proof}

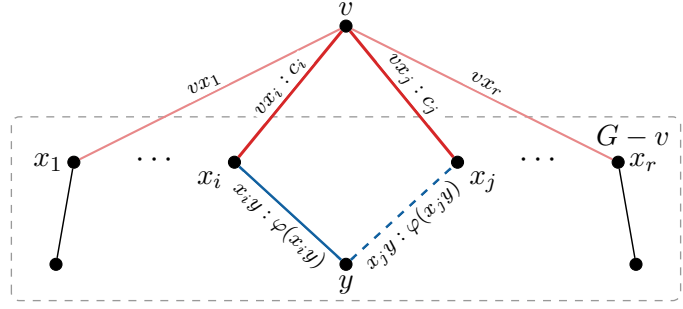
\begin{figure}[htbp]
\centering
\begin{tikzpicture}[x=1.18cm,y=1.0cm]
\draw[rounded corners=3pt,draw=black!40,dashed,line width=.5pt]
  (-3.75,-1.38) rectangle (3.75,1.05);
\node[graphlabel,anchor=north east,fill=white] at (3.68,1.0) {\(G-v\)};

\node[graphvertex,label={[graphlabel]above:\(v\)}] (v) at (0,2.25) {};
\node[graphvertex,label={[graphlabel]left:\(x_1\)}] (x1) at (-3.05,0.45) {};
\node[graphvertex,label={[graphlabel]below left,xshift=-1pt:\(x_i\)}] (xi) at (-1.25,0.45) {};
\node[graphvertex,label={[graphlabel]below right,xshift=1pt:\(x_j\)}] (xj) at (1.25,0.45) {};
\node[graphvertex,label={[graphlabel]right:\(x_r\)}] (xr) at (3.05,0.45) {};
\node[ellipsislabel] at (-2.15,0.48) {\(\cdots\)};
\node[ellipsislabel] at (2.15,0.48) {\(\cdots\)};
\node[graphvertex,label={[graphlabel]below:\(y\)}] (y) at (0,-0.9) {};
\node[graphvertex] (p1) at (-3.25,-0.9) {};
\node[graphvertex] (pr) at (3.25,-0.9) {};

\draw[draw=StructureRed!55,line width=.8pt] (v)--node[graphlabel,sloped,above,font=\scriptsize,fill=white]{\(vx_1\)} (x1);
\draw[focusedge] (v)--node[graphlabel,sloped,above,font=\scriptsize,fill=white]{\(vx_i:c_i\)} (xi);
\draw[focusedge] (v)--node[graphlabel,sloped,above,font=\scriptsize,fill=white]{\(vx_j:c_j\)} (xj);
\draw[draw=StructureRed!55,line width=.8pt] (v)--node[graphlabel,sloped,above,font=\scriptsize,fill=white]{\(vx_r\)} (xr);
\draw[draw=StructureBlue,line width=1.0pt] (xi)--node[graphlabel,sloped,below,font=\scriptsize,fill=white]{\(x_iy:\varphi(x_iy)\)} (y);
\draw[secondedge] (y)--node[graphlabel,sloped,below,font=\scriptsize,fill=white]{\(x_jy:\varphi(x_jy)\)} (xj);
\draw[graphedge] (x1)--(p1);
\draw[graphedge] (xr)--(pr);
\end{tikzpicture}
\caption{Restoring the star at \(v\). The dashed box represents \(G-v\), and the red edges are the newly colored edges \(vx_1,\ldots,vx_r\). The highlighted cycle \(vx_iyx_jv\) is one of the new \(4\)-cycles: the choice \(c_i\) must avoid \(\varphi(x_jy)\), and \(c_j\) must avoid \(\varphi(x_iy)\).}
\label{fig:vertex-extension}
\end{figure}

\Needspace{9\baselineskip}
\begin{lemma}[Forbidden colors at a deleted vertex]
\label{lem:packet-load}
Let \(t\ge2\) be an integer, and let \(G\) be \(K_{2,t}\)-free. For \(1\le i\le r\), the number of colors forbidden for \(vx_i\) is at most
\begin{equation}
d_G(x_i)-1+
\sum_{j\ne i}\min\{t-2,d_G(x_j)-1\}.
\label{eq:packet-compressed}
\end{equation}
In particular, at most \(d_G(x_i)-1+(r-1)(t-2)\) colors are forbidden.
\end{lemma}

\begin{proof}
Properness excludes at most \(d_G(x_i)-1\) colors. For a fixed \(j\ne i\), an old edge opposite \(vx_i\) has the form \(x_jy\), where \(y\in N_G(x_i)\cap N_G(x_j)\setminus\{v\}\). Since \(v\) is already a common neighbor of \(x_i\) and \(x_j\), there are at most \(t-2\) such edges. They are all incident with \(x_j\) in \(G-v\), so they use at most \(d_G(x_j)-1\) colors. Summing the smaller of these bounds over \(j\ne i\) proves \eqref{eq:packet-compressed}; replacing each summand by \(t-2\) gives the final estimate.
\end{proof}

Fix a linear order \(<\) of \(V(G)\). For \(x\in V(G)\), let \(P_1(x)=\{w<x:xw\in E(G)\}\), and let \(P_2(x)\) consist of \(P_1(x)\) together with the vertices \(w<x\) for which \(xu,uw\in E(G)\) for some \(u>x\). For nonnegative integers \(p,q\), the order \(<\) is a \emph{\((p,q)\)-order} if \(|P_1(x)|\le p\) and \(|P_2(x)|\le q\) for every \(x\in V(G)\). The example in \cref{fig:two-step-predecessors} illustrates the difference between the two predecessor sets.

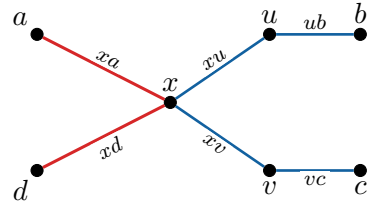
\begin{figure}[htbp]
\centering
\begin{minipage}[c]{0.46\textwidth}
\centering
\small
\begin{tabular}{@{}lll@{}}
\toprule
Role & Vertices & Positions \\
\midrule
\(P_1(x)\) & \(a,d\) & \(1,4\) \\
\(P_2(x)\setminus P_1(x)\) & \(b,c\) & \(2,3\) \\
\(x\) & \(x\) & \(5\) \\
later vertices & \(u,v\) & \(6,7\) \\
\bottomrule
\end{tabular}
\end{minipage}\hfill
\begin{minipage}[c]{0.50\textwidth}
\centering
\begin{tikzpicture}[x=.75cm,y=.9cm]
\node[graphvertex,label={[graphlabel]above:\(x\)}] (x) at (0,1.35) {};
\node[graphvertex,label={[graphlabel]above left:\(a\)}] (a) at (-2.35,2.35) {};
\node[graphvertex,label={[graphlabel]below left:\(d\)}] (d) at (-2.35,0.35) {};
\node[graphvertex,label={[graphlabel]above:\(u\)}] (u) at (1.75,2.35) {};
\node[graphvertex,label={[graphlabel]below:\(v\)}] (vv) at (1.75,0.35) {};
\node[graphvertex,label={[graphlabel]above:\(b\)}] (b) at (3.35,2.35) {};
\node[graphvertex,label={[graphlabel]below:\(c\)}] (c) at (3.35,0.35) {};

\draw[focusedge] (x)--node[graphlabel,sloped,above,font=\scriptsize,fill=white]{\(xa\)} (a);
\draw[focusedge] (x)--node[graphlabel,sloped,below,font=\scriptsize,fill=white]{\(xd\)} (d);
\draw[draw=StructureBlue,line width=1.0pt] (x)--node[graphlabel,sloped,above,font=\scriptsize,fill=white]{\(xu\)} (u);
\draw[draw=StructureBlue,line width=1.0pt] (u)--node[graphlabel,above,font=\scriptsize,fill=white]{\(ub\)} (b);
\draw[draw=StructureBlue,line width=1.0pt] (x)--node[graphlabel,sloped,below,font=\scriptsize,fill=white]{\(xv\)} (vv);
\draw[draw=StructureBlue,line width=1.0pt] (vv)--node[graphlabel,below,font=\scriptsize,fill=white]{\(vc\)} (c);
\end{tikzpicture}
\end{minipage}
\caption{The order positions and the predecessor sets of \(x\). The table encodes the order \(a<b<c<d<x<u<v\) by the positions \(1,\ldots,7\). In the graph, red edges join \(x\) to \(P_1(x)=\{a,d\}\), while blue two-edge paths through the later vertices \(u,v\) reach \(P_2(x)\setminus P_1(x)=\{b,c\}\).}
\label{fig:two-step-predecessors}
\end{figure}

\begin{theorem}[Kierstead--Mohar--{\v S}pacapan--Yang--Zhu~\cite{KiersteadMoharSpacapanYangZhu2009}]
\label{thm:planar-two-step-order}
Every planar graph \(G\) has a \((5,8)\)-order, that is, a linear order \(<\) of \(V(G)\) such that each vertex \(x\) has at most \(5\) neighbors preceding it and at most \(8\) vertices \(w<x\) for which \(xw\in E(G)\) or \(xu,uw\in E(G)\) for some \(u>x\).
\end{theorem}

\subsection{Coloring from a vertex order}

Let \(G\) have a fixed vertex order \(<\). Assign each edge \(xy\), where \(x<y\), to its earlier endpoint and put \(S_x=\{xy\in E(G):x<y\}\). We call the sets \(S_x\) \emph{star blocks} and color them in increasing order of their centers. Before \(S_x\) is colored, every edge whose earlier endpoint precedes \(x\) is colored, and every other edge is uncolored.

The planar two-step order determines the order in which the star blocks are colored. To determine which colors are unavailable for \(S_x\), we first describe every adjacency in \(L^+(G)\) from an earlier block to \(S_x\). These adjacencies split into a bounded-degree part and complete bipartite parts supported on common neighborhoods. Related star-block orders occur in strong edge-coloring~\cite{Wang2014,Yu2015}; here the order controls adjacency in \(L^+(G)\).

\begin{lemma}[Earlier adjacent edges]
\label{lem:earlier-conflict}
Let \(e=xu\in S_x\), and let \(f=wz\in S_w\) be an earlier colored edge, so \(w<x\). If \(e\) and \(f\) are adjacent in \(L^+(G)\) and \(f\ne wx\), then \(w\in P_2(x)\).
\end{lemma}

\begin{proof}
If \(e\) and \(f\) are incident away from \(x\), then \(u=z\). Since \(u>x\), the path \(xuw\) places \(w\) in \(P_2(x)\). If \(e\) and \(f\) are opposite on a \(4\)-cycle, its cyclic order is either \(x-u-w-z-x\) or \(x-u-z-w-x\). In the first case, \(xuw\) is again a two-edge path whose internal vertex follows \(x\). In the second case, \(xw\in E(G)\), so \(w\in P_1(x)\).
\end{proof}

Fix \(w\in P_2(x)\), omit the edge \(wx\) when it exists, and put \(C_{xw}=N_G(x)\cap N_G(w)\). With this notation, the following lemma decomposes the adjacencies in \(L^+(G)\) between the two star blocks. \Cref{fig:two-star-conflicts} shows the two ways in which edges from the blocks can be opposite on a \(4\)-cycle.

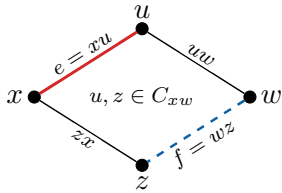
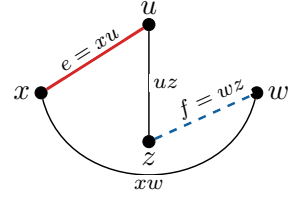
\begin{figure}[!htbp]
\centering
\begin{subfigure}[t]{0.46\textwidth}
\centering
\begin{tikzpicture}[x=.92cm,y=.86cm]
\node[graphvertex,label={[graphlabel]left:\(x\)}] (x) at (-1.55,0) {};
\node[graphvertex,label={[graphlabel]above:\(u\)}] (u) at (0,1.05) {};
\node[graphvertex,label={[graphlabel]right:\(w\)}] (w) at (1.55,0) {};
\node[graphvertex,label={[graphlabel]below:\(z\)}] (z) at (0,-1.05) {};
\draw[focusedge] (x)--node[graphlabel,sloped,above,font=\scriptsize,fill=white]{\(e=xu\)} (u);
\draw[graphedge] (u)--node[graphlabel,sloped,above,font=\scriptsize,fill=white]{\(uw\)} (w);
\draw[secondedge] (w)--node[graphlabel,sloped,below,font=\scriptsize,fill=white]{\(f=wz\)} (z);
\draw[graphedge] (z)--node[graphlabel,sloped,below,font=\scriptsize,fill=white]{\(zx\)} (x);
\node[graphlabel,font=\scriptsize,fill=white,inner sep=1pt] at (0,0) {\(u,z\in C_{xw}\)};
\end{tikzpicture}
\caption{Two \(x\)--\(w\) paths through \(u,z\in C_{xw}\).}
\end{subfigure}
\hfill
\begin{subfigure}[t]{0.46\textwidth}
\centering
\begin{tikzpicture}[x=.92cm,y=.86cm]
\node[graphvertex,label={[graphlabel]left:\(x\)}] (x) at (-1.55,0) {};
\node[graphvertex,label={[graphlabel]above:\(u\)}] (u) at (0,1.05) {};
\node[graphvertex,label={[graphlabel]right:\(w\)}] (w) at (1.55,0) {};
\node[graphvertex,label={[graphlabel]below:\(z\)}] (z) at (0,-.75) {};
\draw[focusedge] (x)--node[graphlabel,sloped,above,font=\scriptsize,fill=white]{\(e=xu\)} (u);
\draw[graphedge] (u)--node[graphlabel,right,font=\scriptsize,fill=white]{\(uz\)} (z);
\draw[secondedge] (z)--node[graphlabel,sloped,above,pos=.65,font=\scriptsize,fill=white]{\(f=wz\)} (w);
\draw[graphedge] (w) .. controls (1.15,-1.65) and (-1.15,-1.65) .. node[graphlabel,below,font=\scriptsize,fill=white]{\(xw\)} (x);
\end{tikzpicture}
\caption{The additional \(4\)-cycle when \(xw\in E(G)\).}
\end{subfigure}
\caption{Opposite-edge adjacencies between \(e=xu\in S_x\) and \(f=wz\in S_w\). In (a), \(u,z\in C_{xw}\) give the complete bipartite part described in \cref{lem:two-star}. In (b), the edges \(xw\) and \(uz\) give the additional \(4\)-cycle \(xuzwx\). The incidence case \(u=z\) and the remaining edges of the two star blocks are omitted.}
\label{fig:two-star-conflicts}
\end{figure}

\FloatBarrier

\begin{lemma}[Adjacencies between two star blocks]
\label{lem:two-star}
Let \(t\ge2\) be an integer, put \(h=t-1\), and let \(G\) be \(K_{2,t}\)-free. The edges of \(L^+(G)\) between \(S_x\) and \(S_w\setminus\{wx\}\) form the union of the following two bipartite graphs.
\begin{enumerate}
\item The complete bipartite graph with parts \(E_{xw}=\{xu\in S_x:u\in C_{xw}\}\) and \(A_{xw}=\{wz\in S_w:z\in C_{xw}\}\). Both sets have size at most \(h\).
\item If \(xw\in E(G)\), an additional bipartite graph whose maximum degree on each side is at most \(h-1\).
\end{enumerate}
\end{lemma}

\begin{proof}
Suppose first that \(xw\notin E(G)\). An edge \(xu\in S_x\) and an edge \(wz\in S_w\) are adjacent in \(L^+(G)\) if and only if \(u=z\), or \(x-u-w-z-x\) is a \(4\)-cycle. Either case occurs exactly when \(u,z\in C_{xw}\), giving the complete bipartite graph with parts \(E_{xw}\) and \(A_{xw}\). Since \(G\) is \(K_{2,t}\)-free, \(|C_{xw}|\le h\).

Now suppose that \(xw\in E(G)\). The complete bipartite graph remains, and the only additional possibility is a \(4\)-cycle \(x-u-z-w-x\).
For a fixed \(u\), the vertices \(x,z\) are common neighbors of \(u,w\). Since \(x\) is already one such neighbor, there are at most \(t-2=h-1\) choices for \(z\). Symmetrically, for a fixed \(z\), the vertices \(w,u\) are common neighbors of \(x,z\), and there are at most \(h-1\) choices for \(u\).
\end{proof}

To color an entire star block at once, we need a matching that avoids both parts of this decomposition. The following Hall-type lemma provides such a matching when the forbidden graph consists of a bounded-degree part and a bounded number of complete bipartite parts of bounded size.

\begin{lemma}[Hall matching lemma]
\label{lem:structured-hall}
Let \(\Gamma\) be a bipartite graph with parts \(R,U\), where \(|R|\le|U|=N\). Let \(B,q\) be nonnegative integers and \(h\) a positive integer. Suppose
\[
\Gamma=\Gamma_0\cup\bigcup_{i=1}^{q}(E_i\times A_i),
\]
where \(\Delta(\Gamma_0)\le B\), \(E_i\subseteq R\), \(A_i\subseteq U\), and \(|E_i|,|A_i|\le h\) for \(1\le i\le q\). If \(N\ge2B\) and \(N\ge B+(q+1)h\), then the bipartite complement of \(\Gamma\) in \(K_{R,U}\) has a matching that saturates \(R\).
\end{lemma}

\begin{proof}
For \(r\in R\), let \(L(r)=U\setminus N_\Gamma(r)\). We verify Hall's condition. Fix a nonempty set \(Z\subseteq R\) and put \(z=|Z|\).

First suppose that \(z\le B\). For \(r\in Z\), let \(m(r)\) be the number of sets \(E_i\) containing \(r\). Since
\[
\sum_{r\in Z}m(r)\le\sum_{i=1}^{q}|E_i|\le qh,
\]
some \(r\in Z\) satisfies \(m(r)\le qh/z\). If \(z\le h\), then
\[
|L(r)|\ge N-B-qh\ge h\ge z.
\]
If \(h<z\le qh\), then
\[
|L(r)|
\ge N-B-\frac{qh^2}{z}
\ge(q+1)h-\frac{qh^2}{z}
\ge z,
\]
where the last inequality is equivalent to \((z-h)(qh-z)\ge0\). Finally, if \(qh<z\le B\), some \(r\in Z\) belongs to none of the sets \(E_i\), and
\[
|L(r)|\ge N-B\ge B\ge z.
\]
Thus \(\left|\bigcup_{r\in Z}L(r)\right|\ge z\) whenever \(z\le B\).

Now suppose that \(z>B\), and let \(C_Z=U\setminus\bigcup_{r\in Z}L(r)\).
For each \(\alpha\in C_Z\), the graph \(\Gamma_0\) covers at most \(B\) vertices of \(Z\), so the complete bipartite blocks containing \(\alpha\) cover at least \(z-B\) vertices. If \(\alpha\) belongs to \(a_\alpha\) of the sets \(A_i\), then \(a_\alpha h\ge z-B\). Put \(a=\lceil(z-B)/h\rceil\).
Every \(\alpha\in C_Z\) belongs to at least \(a\) sets \(A_i\). Since \(\sum_i|A_i|\le qh\), we have \(|C_Z|\le qh/a\). If \(a\le q\), then
\[
\left|\bigcup_{r\in Z}L(r)\right|
\ge N-\frac{qh}{a}
\ge B+(q+1)h-\frac{qh}{a}
\ge B+ah
\ge z,
\]
where \(a+q/a\le q+1\). If \(a>q\), then \(C_Z=\varnothing\), so the union of the lists is \(U\) and has size \(N\ge|R|\ge z\). Hall's condition follows.
\end{proof}

\begin{lemma}[B-coloring from a two-step order]
\label{lem:two-step-coloring}
Let \(t\ge2\), \(p,q\ge0\), and \(0\le\rho\le t-1\) be integers. Suppose that a \(K_{2,t}\)-free graph \(G\) has a \((p,q)\)-order. If
\begin{equation}
\Delta(G)\ge
\max\{(2t-3)p,(p+q+1)(t-1)\}-\rho,
\label{eq:two-step-threshold}
\end{equation}
then \(\qB(G)\le\Delta(G)+\rho\).
\end{lemma}

\begin{proof}
Put \(h=t-1\), use a fixed color set \(C\) of size \(K=\Delta(G)+\rho\), and color the star blocks \(S_x\) in increasing order of their centers. Suppose that all earlier blocks have been colored so that any two colored edges adjacent in \(L^+(G)\) have distinct colors.

Let \(d=|P_1(x)|\). The \(d\) colored edges incident with \(x\) have pairwise distinct colors. Remove these colors from \(C\) and call the remaining set \(U\). Then
\[
N:=|U|=\Delta(G)+\rho-d,
\qquad
|S_x|=d_G(x)-d\le N.
\]

Construct a bipartite forbidden graph with parts \(S_x\) and \(U\), joining \(e\in S_x\) to \(\alpha\in U\) when \(\alpha\) appears on an earlier edge adjacent to \(e\) in \(L^+(G)\). By \cref{lem:earlier-conflict}, only centers in \(P_2(x)\) occur. For each such center \(w\), let \(\widehat A_{xw}\subseteq U\) be the set of colors in \(U\) used on the edges in \(A_{xw}\) defined in \cref{lem:two-star}. Since \(S_w\) is a properly colored star, \(|\widehat A_{xw}|\le|A_{xw}|\le h\). The complete bipartite part in \cref{lem:two-star} therefore gives the forbidden block \(E_{xw}\times\widehat A_{xw}\), and there are at most \(q\) such blocks.

For each of the \(d\) direct predecessors, the additional bipartite graph in \cref{lem:two-star} has maximum degree at most \(h-1\) on the edge side. Its maximum degree on the color side is also at most \(h-1\), because each color occurs on at most one edge of an earlier star block. Their union consequently has maximum degree at most \(B=d(h-1)\).

The two numerical hypotheses of \cref{lem:structured-hall} follow from \eqref{eq:two-step-threshold}. Indeed,
\[
N\ge\Delta(G)+\rho-p
\ge2p(h-1)
\ge2d(h-1)=2B,
\]
and
\[
\begin{aligned}
N
&\ge\Delta(G)+\rho-d\\
&\ge(p+q+1)h-d\\
&\ge d(h-1)+(q+1)h\\
&=B+(q+1)h.
\end{aligned}
\]
By \cref{lem:structured-hall}, the bipartite complement of the forbidden graph has a matching saturating \(S_x\). Color the block according to this matching. Any two colored edges adjacent in \(L^+(G)\) still have distinct colors, and induction over the star blocks gives a B-coloring of \(G\).
\end{proof}

Applying the two-step coloring lemma to the \((5,8)\)-order of a planar graph gives the following bound.

\begin{lemma}[Planar coloring bound]
\label{lem:planar-tradeoff}
Let \(t\ge2\) and \(0\le\rho\le t-1\) be integers, and let \(G\) be a \(K_{2,t}\)-free planar graph. If \(\Delta(G)\ge14(t-1)-\rho\), then \(\qB(G)\le\Delta(G)+\rho\).
\end{lemma}

\begin{proof}
Apply \cref{thm:planar-two-step-order,lem:two-step-coloring} with \((p,q)=(5,8)\) and \(h=t-1\).
Since
\[
(2h-1)p=10h-5<14h=(p+q+1)h,
\]
condition \eqref{eq:two-step-threshold} becomes \(\Delta(G)\ge14(t-1)-\rho\), as required.
\end{proof}

Taking \(\rho=t-1\) and \(\rho=0\), respectively, gives
\begin{equation}
\qB(G)\le
\begin{cases}
\Delta(G)+t-1, & \text{if }\Delta(G)\ge13(t-1),\\[2pt]
\Delta(G), & \text{if }\Delta(G)\ge14(t-1).
\end{cases}
\label{eq:target-from-tradeoff}
\end{equation}
In the second case, \(\qB(G)=\Delta(G)\), since \(\qB(G)\ge\Delta(G)\).

\begin{proof}[Proof of Theorem~A]
If \(t=2\), then \(G\) has no \(4\)-cycle, so \(\qB(G)=\chi'(G)\). The theorems of Sanders and Zhao and of Zhang give \(\chi'(G)=\Delta(G)\) for planar graphs with \(\Delta(G)\ge7\)~\cite{SandersZhao2001,Zhang2000}. For \(t\ge3\), the second case above applies.
\end{proof}

The following example shows why the lower bound on \(\Delta(G)\) in Theorem~A must depend on \(t\).

\begin{proposition}[A degree obstruction to equality]
\label{prop:equality-threshold-lower}
For every integer \(t\ge2\), there is a planar \(K_{2,t}\)-free graph \(L_t\) with \(\Delta(L_t)=2t-2\) and \(\qB(L_t)\ge2t-1\). Consequently, a sufficient lower bound on \(\Delta(G)\) that forces \(\qB(G)=\Delta(G)\) for all \(K_{2,t}\)-free planar graphs must be at least \(2t-2\) in general and cannot be independent of \(t\).
\end{proposition}

\begin{proof}
Construct \(L_t\) from adjacent vertices \(a,b\) by adding \(t-1\) common neighbors \(z_1,\ldots,z_{t-1}\) and attaching \(t-2\) leaves to \(a\), as shown in \cref{fig:lower-bound-graph}. This graph is planar and \(K_{2,t}\)-free, and \(\Delta(L_t)=2t-2\). The \(2t-1\) edges in \(\{ab\}\cup\{az_i,bz_i:1\le i\le t-1\}\) form a clique in \(L^+(L_t)\). Edges at a common endpoint are incident, and for distinct \(i,j\), the edges \(az_i\) and \(bz_j\) are opposite on the \(4\)-cycle \(az_ibz_ja\). Hence \(\qB(L_t)\ge2t-1\).
\end{proof}

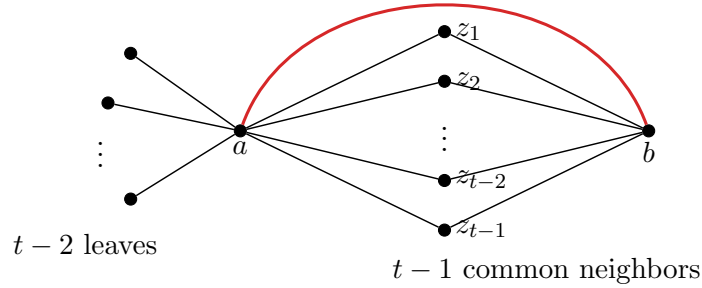
\begin{figure}[htbp]
\centering
\begin{tikzpicture}[x=1.0cm,y=.82cm]
\node[graphvertex,label={[graphlabel]below:\(a\)}] (a) at (-2.7,0) {};
\node[graphvertex,label={[graphlabel]below:\(b\)}] (b) at (2.7,0) {};
\node[graphvertex,label={[graphlabel]right:\(z_1\)}] (z1) at (0,1.6) {};
\node[graphvertex,label={[graphlabel]right:\(z_2\)}] (z2) at (0,.8) {};
\node[ellipsislabel] at (0,0) {\(\vdots\)};
\node[graphvertex,label={[graphlabel]right:\(z_{t-2}\)}] (ztwo) at (0,-.8) {};
\node[graphvertex,label={[graphlabel]right:\(z_{t-1}\)}] (zone) at (0,-1.6) {};
\foreach \z in {z1,z2,ztwo,zone}{
  \draw[graphedge] (a)--(\z);
  \draw[graphedge] (b)--(\z);
}
\draw[focusedge] (a) to[bend left=70,looseness=1.08] (b);
\node[graphvertex] (w1) at (-4.15,1.25) {};
\node[graphvertex] (w2) at (-4.45,.45) {};
\node[ellipsislabel] at (-4.55,-.25) {\(\vdots\)};
\node[graphvertex] (w3) at (-4.15,-1.1) {};
\draw[graphedge] (a)--(w1);
\draw[graphedge] (a)--(w2);
\draw[graphedge] (a)--(w3);
\node[graphlabel,align=center] at (-4.75,-1.85) {\(t-2\) leaves};
\node[graphlabel,align=left] at (1.35,-2.25) {\(t-1\) common neighbors};
\end{tikzpicture}
\caption{The planar graph \(L_t\). The red edge \(ab\), together with the \(2(t-1)\) edges joining \(a,b\) to their common neighbors, corresponds to a clique of order \(2t-1\) in \(L^+(L_t)\).}
\label{fig:lower-bound-graph}
\end{figure}

\section{Proof of Theorem~B}
\label{sec:theorem-b}

We first prove a fixed-\(t\) form of Theorem~B that is sharper than its uniform statement.

\begin{theorem}
\label{thm:fixed-t-thresholds}
Let \(t\ge2\) be an integer, and let \(G\) be a \(K_{2,t}\)-free planar graph. Then \(\qB(G)\le\Delta(G)+t-1\) whenever \(t=2\) or \(t\ge35\), and also whenever \(3\le t\le34\) and \(\Delta(G)\ge13(t-1)\).
\end{theorem}

The bound in \eqref{eq:target-from-tradeoff} proves the range \(3\le t\le34\). A result of Vuolo~\cite{Vuolo2026}, which appeared during the preparation of this manuscript, gives the bound for \(t\ge39\) without a lower bound on \(\Delta(G)\); the proof below covers all \(t\ge35\) and does not use this result.

\subsection{The case \texorpdfstring{\(t\ge35\)}{t >= 35}}

For \(t\ge35\), we use the stars-and-bunches structure of plane graphs~\cite{BorodinBroersmaGlebovVandenHeuvel2001,KongWangZheng2026}. In a plane graph \(G\), a \emph{bunch} \(B(x,y;m)\) is a maximal sequence \(Q_1,\ldots,Q_m\) of paths of length one or two between two poles \(x,y\), consecutive around both poles, such that \(Q_i\cup Q_{i+1}\) is a nonseparating cycle for \(1\le i<m\). If \(Q_i=xz_iy\), then \(z_i\) is a \emph{brother}; if \(Q_i=xy\), then \(xy\) is a \emph{parental edge}. A brother \(z_i\) is \emph{internal} when \(2\le i\le m-1\), and we call \(xz_i\) and \(yz_i\) its pole edges. Every internal brother has degree \(2\), \(3\), or \(4\) and is adjacent only to \(x,y\) and possibly its neighboring brothers, as illustrated in \cref{fig:bunch}.

\begin{figure}[htbp]
\centering
\begin{subfigure}[t]{0.46\textwidth}
\centering
\begin{tikzpicture}[x=.8cm,y=.8cm]
\node[graphvertex,label={[graphlabel]above:\(x\)}] (x) at (0,2) {};
\node[graphvertex,label={[graphlabel]below:\(y\)}] (y) at (0,-2) {};
\foreach \i/\xx in {1/-2.2,2/-1.1,3/0,4/1.1,5/2.2}{
  \node[graphvertex,label={[graphlabel]above right:\(z_{\i}\)}] (z\i) at (\xx,0) {};
  \draw[graphedge] (x)--(z\i)--(y);
}
\draw[optionaledge] (z2)--(z3);
\draw[optionaledge] (z3)--(z4);
\end{tikzpicture}
\caption{A bunch with brother paths only.}
\end{subfigure}
\hfill
\begin{subfigure}[t]{0.46\textwidth}
\centering
\begin{tikzpicture}[x=.8cm,y=.8cm]
\node[graphvertex,label={[graphlabel]above:\(x\)}] (x) at (0,2) {};
\node[graphvertex,label={[graphlabel]below:\(y\)}] (y) at (0,-2) {};
\foreach \i/\xx in {1/-2.2,2/-1.1,3/1.1,4/2.2}{
  \node[graphvertex,label={[graphlabel]above right:\(z_{\i}\)}] (z\i) at (\xx,0) {};
  \draw[graphedge] (x)--(z\i)--(y);
}
\draw[focusedge] (x)--(y);
\draw[optionaledge] (z1)--(z2);
\draw[optionaledge] (z3)--(z4);
\end{tikzpicture}
\caption{A bunch containing the parental edge \(xy\).}
\end{subfigure}
\caption{The bunch \(B(x,y;m)\). Each solid two-edge path through a brother belongs to the bunch; dashed edges indicate possible adjacencies between consecutive brothers. The parental edge, when present, is highlighted in red.}
\label{fig:bunch}
\end{figure}
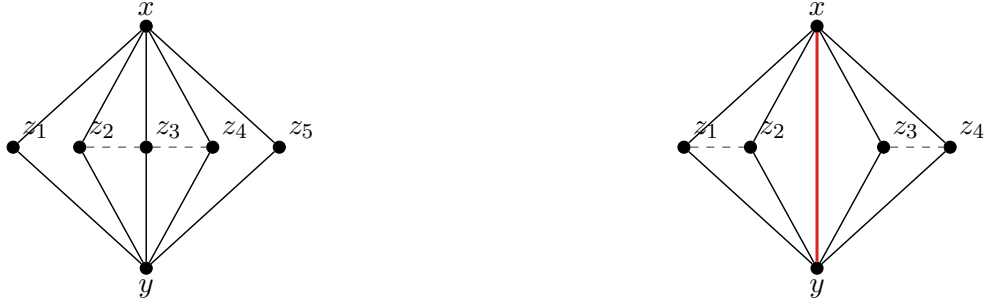

\Needspace{7\baselineskip}
\begin{lemma}[Stars and bunches~\cite{BorodinBroersmaGlebovVandenHeuvel2001,KongWangZheng2026}]
\label{lem:stars-bunches}
Let \(G\) be a plane graph with \(\delta(G)\ge2\). Then \(G\) contains one of the following configurations.
\begin{enumerate}[label=\textup{B\arabic*.}]
\item A \(k\)-vertex \(v\), where \(2\le k\le5\), with neighbors \(x_1,\ldots,x_k\) such that \(d_G(x_i)\le25\) for \(1\le i<k\) and \(\sum_{i=1}^{k-1}d_G(x_i)\le38\).
\item A bunch \(B(x,y;m)\) with \(d_G(x)\ge26\) and \(m\ge d_G(x)/5\).
\end{enumerate}
\end{lemma}
The proof uses vertex deletion, so we retain one fixed set of \(D+t-1\) colors. We therefore prove the following stronger form.

\begin{lemma}
\label{lem:large-t-fixed-palette}
Let \(t\ge35\) be an integer, let \(G\) be a \(K_{2,t}\)-free planar graph, and let \(D\ge\Delta(G)\) be an integer. Then \(G\) has a B-coloring from any fixed set of \(K=D+t-1\) colors.
\end{lemma}

\begin{proof}
Fix a set \(C\) of \(K=D+t-1\) colors. If \(D\le28\), then the bound of Kong, Wang, and Zheng~\cite{KongWangZheng2026} gives \(\qB(G)\le2\Delta(G)+6\le D+34\le K\). Hence assume that \(D\ge29\). Suppose that the assertion fails, and choose a counterexample \(G\) with the fewest vertices and, subject to this, the most edges. All graphs considered here are \(K_{2,t}\)-free and planar, have maximum degree at most \(D\), and use the same color set \(C\). By minimality, \(G\) is connected and has no isolated vertex. Fix a plane embedding of \(G\).

If \(v\) has degree one, color \(G-v\) by minimality and then color its incident edge; at most \(D-1\) colors are forbidden. If \(v\) has degree two, then \cref{lem:packet-load} gives at least \(K-(D-1)-(t-2)=2\) colors for each of its incident edges. By \cref{lem:three-source}, distinct choices extend the coloring. Thus \(\delta(G)\ge3\), and we apply \cref{lem:stars-bunches}.

\smallskip
\noindent\emph{Configuration B1.}
Let \(v\) be a \(k\)-vertex with neighbors \(x_1,\ldots,x_{k-1},y\) as in \cref{lem:stars-bunches}, and put \(S=\sum_{i=1}^{k-1}d_G(x_i)\le38\). Since \(\delta(G)\ge3\), we have \(3\le k\le5\).

Choose two of \(x_1,\ldots,x_{k-1}\), say \(a,b\), that are consecutive around \(v\). Such a pair exists because \(y\) is the only other neighbor. If \(ab\notin E(G)\), we can add \(ab\) in the face incident with the angle \(avb\). The new degrees of \(a,b\) are at most \(26\le D\). Only pairs containing \(a\) or \(b\) can gain a common neighbor, and each such pair has at most \(26<t\) common neighbors in \(G+ab\). Thus \(G+ab\) remains \(K_{2,t}\)-free and planar, with maximum degree at most \(D\). Any B-coloring of \(G+ab\) restricts to one of \(G\), so \(G+ab\) would be a counterexample with the same number of vertices and more edges. Consequently, \(ab\in E(G)\).

Color \(G-v\) by minimality. The sum \(\sum_{i=1}^{k-1}(d_G(x_i)-1)=S-(k-1)\) counts every old edge opposite to \(vy\) at least once, but overcounts by at least two. Indeed, if neither \(a\) nor \(b\) is adjacent to \(y\), then \(ab\) is counted twice and is not opposite to \(vy\). Otherwise, say \(ay\in E(G)\), the edge \(ab\) is counted twice but contributes at most one distinct opposite edge, while \(ay\) is counted once and is incident to \(vy\). Hence at most
\[
(d_G(y)-1)+S-(k-1)-2\le D+36-k
\]
colors are forbidden for \(vy\). At least \(K-(D+36-k)=t+k-37\ge1\) color remains, so color \(vy\) first.

For \(1\le i<k\), the local estimate of Kong, Wang, and Zheng~\cite[Lemma~4]{KongWangZheng2026} gives at most \(2d_G(x_i)+k-3\) forbidden colors for \(vx_i\) before any edge at \(v\) is colored. Since \(D\ge29\), \(t\ge35\), and \(d_G(x_i)\le25\), each initial list has at least \(16-k\ge k\) colors. Color \(vx_1,\ldots,vx_{k-1}\) successively. At each step, at most \(k-1\) colors have already been used at \(v\), so \cref{lem:three-source} completes the extension.

\smallskip
\noindent\emph{Configuration B2.}
Let \(B(x,y;m)\) be the bunch. Since \(d_G(x)\ge26\) and \(m\ge d_G(x)/5\), we have \(m\ge6\). Put \(s=|N_G(x)\cap N_G(y)|\le t-1\). We delete an internal brother and color the remaining graph by minimality. The recoloring that uses a distant pole-edge color below also occurs in Vuolo's proof~\cite[Case~2a]{Vuolo2026}; we give the details with the present color set.

For each coloring \(\varphi\) obtained after deleting a brother, write \(a_i=\varphi(xz_i)\) and \(b_i=\varphi(yz_i)\) for the remaining brothers. All colored pole edges have distinct colors, since any two are incident or opposite on a \(4\)-cycle. Deleting a vertex removes no adjacency in \(L^+(G)\) between edges that remain, so these colorings also respect the adjacencies of the original graph. An edge joining two internal brothers has at most six incident edges and nine opposite edges, hence at most \(15\) neighbors in \(L^+(G)\). Since \(K\ge63\), once the two pole edges at the deleted brother are colored, its remaining incident edges can be colored greedily.

\smallskip
\noindent\emph{No parental edge and \(m\ge7\).}
Delete \(v=z_4\), and color \(H=G-v\) by minimality. Apart from the old edges incident with \(x\) and the edges \(yu\) with \(u\in N_G(x)\cap N_G(y)\setminus\{v\}\), the only old edges that can be opposite to \(xv\) are \(z_2z_3\) and \(z_5z_6\). The analogous assertion holds for \(yv\).

If \(z_2z_3\) is present, recolor it with one of \(a_6,b_6\) that differs from the color of \(z_1z_2\), when that edge is present. This is possible because \(a_6\ne b_6\). In \(H\), every edge adjacent to \(z_2z_3\) in \(L^+(H)\) is a pole edge at \(z_1,z_2,z_3\), the edge \(xy\) if present, or \(z_1z_2\). Both candidate colors differ from all these pole-edge colors and from the color of \(xy\). Thus the recoloring is proper in \(L^+(H)\). Similarly, recolor \(z_5z_6\), if present, with one of \(a_2,b_2\) avoiding the color of \(z_6z_7\), when present. The two recolored edges are not adjacent in \(L^+(G)\), so these choices do not interfere. These choices are illustrated in \cref{fig:b2-window}.

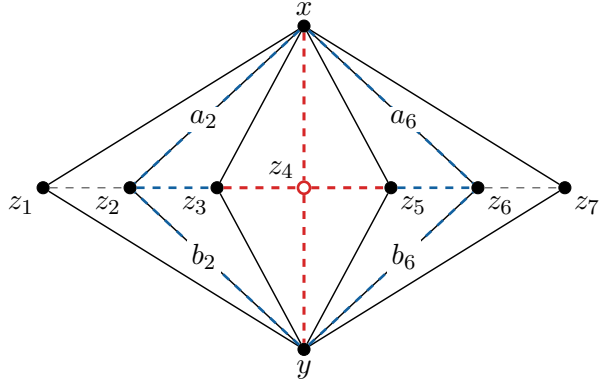
\begin{figure}[htbp]
\centering
\begin{tikzpicture}[x=1.15cm,y=.95cm]
\node[graphvertex,label={[graphlabel]above:\(x\)}] (x) at (0,2.25) {};
\node[graphvertex,label={[graphlabel]below:\(y\)}] (y) at (0,-2.25) {};
\foreach \i/\xx in {1/-3,2/-2,3/-1}{
  \node[graphvertex,label={[graphlabel]below left:\(z_{\i}\)}] (z\i) at (\xx,0) {};
  \draw[graphedge] (x)--(z\i)--(y);
}
\foreach \i/\xx in {5/1,6/2,7/3}{
  \node[graphvertex,label={[graphlabel]below right:\(z_{\i}\)}] (z\i) at (\xx,0) {};
  \draw[graphedge] (x)--(z\i)--(y);
}
\node[graphvertex,fill=white,draw=StructureRed,line width=1pt,label={[graphlabel]above left:\(z_4\)}] (z4) at (0,0) {};
\draw[focusedge,dashed] (x)--(z4)--(y);
\draw[focusedge,dashed] (z3)--(z4)--(z5);
\draw[secondedge] (x)--node[graphlabel,pos=.58,fill=white]{\(a_2\)} (z2);
\draw[secondedge] (y)--node[graphlabel,pos=.58,fill=white]{\(b_2\)} (z2);
\draw[secondedge] (x)--node[graphlabel,pos=.58,fill=white]{\(a_6\)} (z6);
\draw[secondedge] (y)--node[graphlabel,pos=.58,fill=white]{\(b_6\)} (z6);
\draw[optionaledge] (z1)--(z2);
\draw[secondedge] (z2)--(z3);
\draw[secondedge] (z5)--(z6);
\draw[optionaledge] (z6)--(z7);
\end{tikzpicture}
\caption{Deleting \(v=z_4\) in a bunch with at least seven paths. The red vertex and edges are restored after coloring \(G-v\); the edges \(vz_3,vz_5\) may be absent. If present, \(z_2z_3\) and \(z_5z_6\) are recolored with \(a_6\) or \(b_6\), and \(a_2\) or \(b_2\), respectively, avoiding the colors of \(z_1z_2\) and \(z_6z_7\), respectively, when those edges are present.}
\label{fig:b2-window}
\end{figure}

Either color on \(xz_i,yz_i\), for \(i\in\{2,6\}\), is forbidden for both \(xv\) and \(yv\). The recolored edges therefore add no further forbidden colors. At most \((D-1)+(s-1)\le D+t-3=K-2\) colors are forbidden for either pole edge. Choose distinct colors for \(xv,yv\), and then color the remaining edges at \(v\) greedily.

\smallskip
\noindent\emph{No parental edge and \(m=6\).}
Here \(d_G(x)\le30\). Delete \(v=z_3\) and color \(G-v\). The old edges that forbid colors for either pole edge consist of the edges at its pole, the opposite pole edges through common neighbors of \(x,y\), and at most two edges joining consecutive brothers. Thus the two forbidden-color counts are at most \(d_G(x)+s\) and \(d_G(y)+s\), respectively. Since \(s\le d_G(x)\le30\), both counts are at most \(D+30\), leaving at least \(t-31\ge4\) colors for each pole edge. Choose distinct colors, and then color the remaining edges at \(v\) greedily.

\smallskip
\noindent\emph{A parental edge.}
Index the bunch so that \(Q_{r+1}=xy\), and reverse the order if necessary so that \(r\ge m-r-1\). Then \(r\ge3\). Delete \(v=z_r\) and color \(H=G-v\). Put \(w=z_{r-1}\) and \(z=z_{r-2}\). Since \(v\) is internal and \(\delta(G)\ge3\), we have \(N_G(v)=\{x,y,w\}\). For \(xv\), the old incident edges and the opposite pole edges forbid at most \((d_G(x)-1)+(s-1)=d_G(x)+s-2\) colors. The only possible additional old edge is \(zw\). The corresponding count for \(yv\) is at most \(d_G(y)+s-2\), again with only \(zw\) possibly additional.

Suppose first that \(r\ge4\). If \(zw\) is present, choose a brother \(u\notin\{z_{r-3},z,w,v\}\); there are at least \(m-1\ge5\) brothers. Recolor \(zw\) with one of \(\varphi(xu),\varphi(yu)\) avoiding the color of \(z_{r-3}z\), when that edge is present. As in the preceding long-bunch case, the old neighbors of \(zw\) in \(L^+(H)\) are pole edges at \(z_{r-3},z,w\), the parental edge, and possibly \(z_{r-3}z\), so this recoloring is valid. Its color is already forbidden for both \(xv\) and \(yv\). Whether or not \(zw\) is present, each pole edge now has at least \(K-(D+s-2)\ge2\) available colors. Color them distinctly and then color \(vw\).

It remains to consider \(r=3\). Now \(m\le7\), so \(d_G(x)\le35\). Each pole edge has at least \(K-(D+s-1)\ge1\) available color. If their lists admit distinct choices, color both edges and then \(vw\). Otherwise, the two lists are the same singleton. Equality must hold throughout the forbidden-color estimates, giving
\[
d_G(x)=d_G(y)=D,\qquad s=t-1,\qquad zw\in E(G).
\]
Moreover, the color of \(zw\) is not among the colors already forbidden by pole edges for either list. Since \(xy\in E(G)\), we have \(D\ge s+1=t\). Together with \(t\ge35\) and \(D=d_G(x)\le35\), this gives \(t=D=35\), \(s=34\), and \(K=69\).

Put \(T=N_G(x)\cap N_G(y)\). We have \(N_G(x)=T\cup\{y\}\) and \(N_G(y)=T\cup\{x\}\). Every vertex of \(G[T]\) has degree at most two; otherwise that vertex, together with \(x,y\) and three of its neighbors in \(T\), would give a \(K_{3,3}\) subgraph. Since \(w\in N_G(z)\cap T\), the vertex \(z\) has at most one further neighbor in \(T\). Also \(N_G(w)=\{x,y,z,v\}\). Hence \(zw\) is incident with at most \(d_G(z)+2\le37\) other edges. Its opposite edges are among \(xy,xv,yv\) and, if \(z\) has another neighbor \(u\) in \(T\), \(xu,yu\). Consequently,
\[
d_{L^+(G)}(zw)\le37+5=42<69.
\]
Temporarily uncolor \(zw\). Both pole-edge lists now have at least two colors, so color \(xv,yv\) distinctly. Restore \(zw\) while avoiding all its colored neighbors in \(L^+(G)\); at least \(69-42\) colors remain. Finally, \(vw\) has at most \((3-1)+(4-1)+(3-1)(4-1)=11\) neighbors in \(L^+(G)\), so it can also be colored. This completes B2 and contradicts the choice of \(G\).
\end{proof}

\begin{proof}[Proof of \cref{thm:fixed-t-thresholds}]
Put \(\Delta=\Delta(G)\). When \(t=2\), the graph has no \(4\)-cycle, so Vizing's theorem gives \(\qB(G)=\chi'(G)\le\Delta+1=\Delta+t-1\)~\cite{Vizing1964}. The range \(3\le t\le34\) follows from \eqref{eq:target-from-tradeoff}. When \(t\ge35\), apply \cref{lem:large-t-fixed-palette} with \(D=\Delta\).
\end{proof}

\enlargethispage{\baselineskip}
\begin{proof}[Proof of Theorem~B]
Put \(\Delta=\Delta(G)\). Since \(\Delta>428\), we have \(\Delta\ge429=13(34-1)\ge13(t-1)\) for \(3\le t\le34\). For \(t=2\) or \(t\ge35\), \cref{thm:fixed-t-thresholds} requires no lower bound on \(\Delta\). Thus \(\qB(G)\le\Delta+t-1\) for every \(t\ge2\).
\end{proof}

The value \(35\) separates the two arguments above; we do not claim that it is best possible.

Let \(t\ge2\), and let \(G\) be a \(K_{2,t}\)-free planar graph with maximum degree \(\Delta>428\). Theorem~B and the \(2\Delta\) bound of Kong, Wang, and Zheng give
\[
\qB(G)\le\Delta+\min\{t-1,\Delta\}.
\]
Indeed, Theorem~B gives \(\qB(G)\le\Delta+t-1\), and \(\Delta>38\) permits the bound \(\qB(G)\le2\Delta\)~\cite{KongWangZheng2026}; take the smaller of the two.

\section{Degenerate graphs}
\label{sec:degenerate}

The proof for planar graphs uses planarity to bound the colors forbidden by previously colored star blocks. The vertex-extension lemma itself requires no embedding and gives the following bound for degenerate graphs. The upper bound also follows from a more general result of Hu, Kong, and Wang~\cite{HuKongWang2026} in a preprint that appeared during the preparation of this manuscript.

\begin{theorem}[Degenerate bound]
\label{thm:degenerate-bound}
Let \(k\ge1\) and \(t\ge2\) be integers, and let \(G\) be a \(k\)-degenerate \(K_{2,t}\)-free graph. Then
\begin{equation}
\qB(G)\le\Delta(G)+(k-1)\min\{t-1,\Delta(G)\}.
\label{eq:degenerate-bound}
\end{equation}
If \(k\ge2\) and \(t-1\ge k\), then equality is attained by \(K_{k,t-1}\).
\end{theorem}

\begin{proof}[Proof of \cref{thm:degenerate-bound}]
Put \(\Delta=\Delta(G)\), \(s=\min\{t-1,\Delta\}\), and \(K=\Delta+(k-1)s\).
We induct on \(|V(G)|\), always using the same \(K\)-element color set. Since \(G\) is \(k\)-degenerate, it has a vertex \(v\) of degree \(r\le k\). By induction, \(G-v\) has a B-coloring with at most \(K\) colors.

Write \(N_G(v)=\{x_1,\ldots,x_r\}\), and form the lists \(L_1,\ldots,L_r\) used in \cref{lem:three-source}. For distinct \(i,j\) with \(1\le i,j\le r\), we have \(\min\{t-2,d_G(x_j)-1\}\le s-1\).
Thus \cref{lem:packet-load} gives
\[
\begin{aligned}
|L_i|
&\ge K-(\Delta-1)-(r-1)(s-1)\\
&=(k-r)s+r\\
&\ge r.
\end{aligned}
\]
We may therefore choose pairwise distinct colors \(c_i\in L_i\) greedily. By \cref{lem:three-source}, these choices extend the coloring to \(G\).

Now suppose that \(k\ge2\) and \(t-1\ge k\), and take \(G=K_{k,t-1}\). This graph is \(k\)-degenerate and \(K_{2,t}\)-free, with \(\Delta=t-1\). Any two distinct edges are either incident or opposite on a \(4\)-cycle, so \(L^+(G)\) is complete. Consequently,
\[
\qB(G)=|E(G)|=k(t-1)
=\Delta+(k-1)(t-1),
\]
which is equality in \eqref{eq:degenerate-bound}.
\end{proof}
The equality example also determines when a lower bound on \(\Delta(G)\) independent of \(t\) can guarantee \eqref{eq:target-palette} throughout the \(k\)-degenerate class.

\begin{corollary}[Lower bounds on \(\Delta\) independent of \(t\)]
\label{cor:degenerate-classification}
For \(k\in\{1,2\}\), every \(k\)-degenerate \(K_{2,t}\)-free graph \(G\) satisfies \(\qB(G)\le\Delta(G)+t-1\) for every integer \(t\ge2\), with no restriction on \(\Delta(G)\). For every integer \(k\ge3\), there is no constant \(M_k\), independent of \(t\), such that all \(k\)-degenerate \(K_{2,t}\)-free graphs \(G\) with \(\Delta(G)>M_k\) satisfy this bound.
\end{corollary}

\begin{proof}
Put \(\Delta=\Delta(G)\). For \(k\in\{1,2\}\), \cref{thm:degenerate-bound} gives \(\qB(G)\le\Delta+(k-1)\min\{t-1,\Delta\}\le\Delta+t-1\). For \(k\ge3\), take arbitrarily large integers \(t\ge k+1\) and let \(G=K_{k,t-1}\). As shown in the proof of \cref{thm:degenerate-bound}, \(\Delta=t-1\) tends to infinity and \(\qB(G)=k(t-1)>2(t-1)=\Delta+t-1\).
\end{proof}

For fixed integers \(k\ge3\) and \(t\ge3\), \cref{thm:degenerate-bound} does not decide whether \(\qB(G)\le\Delta(G)+t-1\) holds throughout the \(k\)-degenerate class once \(\Delta(G)\) is sufficiently large: the extremal graph \(K_{k,t-1}\) has bounded maximum degree when \(k,t\) are fixed. The planar proof succeeds at degeneracy five because planarity supplies the \((5,8)\)-order in \cref{thm:planar-two-step-order}.

\section*{Declaration of AI usage}

During the preparation of this manuscript, the author used ChatGPT 5.6 Sol to explore potential approaches to parts of the mathematical arguments and to improve the language and presentation of the paper. The initial manuscript was completed on July 25, 2026. 

All AI-assisted arguments were carefully checked, revised, and independently verified by the author. The author takes full responsibility for the correctness, originality, and final content of the manuscript.

\printbibliography

\end{document}